\documentclass[12pt, a4paper]{article}
\usepackage{amsmath,amscd}
\usepackage{amstext}
\usepackage{amssymb,epsfig}
\usepackage{amsthm}
\usepackage{amsfonts}
\usepackage{graphicx}
\usepackage{psfrag}
\usepackage[all]{xy}
\usepackage{tikz}
\usetikzlibrary{arrows,automata}

\renewcommand{\footnote}{\endnote}

\newtheorem{theorem}{Theorem}[section]

\newtheorem*{Theorem}{Theorem}
\newtheorem{lemma}[theorem]{Lemma}
\newtheorem{proposition}[theorem]{Proposition}
\newtheorem{corollary}[theorem]{Corollary}

\theoremstyle{definition}

\newtheorem{remark}[theorem]{Remark}

\newtheorem{example}[theorem]{Example}

\usepackage[utf8]{inputenc}      % pour les accents

\begin{document}

\title{Symbolic dynamics of a piecewise rotation:\\ Case of the non symmetric bijective maps}

\author{Nicolas B\'edaride\footnote{ Aix Marseille Université, CNRS, Centrale Marseille, I2M, UMR 7373, 13453 Marseille, France. Email: nicolas.bedaride@univ-amu.fr}\and  Idrissa Kabor\'e\footnote{Universit\'e Polytechnique de Bobo Dioulasso,01 PB 1091, Bobo-Dioulasso, Burkina Faso. Email: ikaborei@yahoo.fr}}

\date{}

\maketitle

\begin{abstract}
We consider a specific piecewise rotation of the plane that is continuous on two half-planes, as studied by some authors like Boshernitzan, Goetz and Quas. If the angle belongs to the set $\{\frac{\pi}{2},\frac{2\pi}{3},\frac{\pi}{4}\}$, we give a complete description of the symbolic dynamics of this map in the non symmetric bijective case.
\end{abstract}

\section{Introduction}

Consider a line $l$ in $\mathbb{R}^2$; it splits the plane on two half-planes. Now, we define a piecewise isometry $T$ on $\mathbb{R}^2$ such that the restriction to each half-plane is given by a rotation. The two rotations are of the same angle with different centers named $O_1, O_2$, which may lay outside of the corresponding half-plane. Without loss of generality, we can identify the plane with the set of complex numbers $\mathbb{C}$ and the line with the real axis $\mathbb{R}$. Then if the centers have coordinates $z_1, z_2$, the map is given by:
$$\begin{array}{ccc}
\mathbb{C}\setminus\mathbb{R}&\rightarrow&\mathbb{C}\setminus\mathbb{R}\\
z&\mapsto&T(z)=\begin{cases}e^{2i\pi \theta}(z-z_1)+z_1\quad Im(z)> 0\\ e^{2i\pi \theta}(z-z_2)+z_2\quad Im(z)<0 \end{cases}
\end{array}
$$

This map is called a {\bf piecewise rotation}. Up to date it is perhaps the piecewise isometry which has been studied the most. 
This map can be of three forms: either it is bijective or non injective or non surjective.
In \cite{Bosh.Goet.03}, Boshernitzan and Goetz show that in the two last cases the map is either globally attractive or globally repulsive. In the bijective case, this map can be written in the following form (see \cite{Goet.Quas.09}), we will use during the rest of the paper. 
The number  $\sigma$ is a real number, such that up to an homothety the images of the origin by the two rotations are at distance $2$. 
$$T_{\theta,\sigma}(z)=\begin{cases}e^{2i\pi \theta}(z+\sigma+1)\quad Im(z)> 0\\ e^{2i\pi \theta}(z+\sigma-1)\quad Im(z)<0 \end{cases}$$
The parameter $\theta$ is called the angle of the map by a slight abuse of notation. 

Goetz and Quas have shown that for a rational angle every orbit is bounded (assuming that the second parameter is small), see \cite{Goet.Quas.09}. 
In order to prove this result they introduce symbolic dynamics for this map. One code the orbit of a point $z$ on a two-letters alphabet according to the half plane where is $T^n_{\theta,\sigma}(z)$. It defines a subshift of $\{1,2\}^\mathbb N$. 

In the present paper we want to give a precise description of the language of this subshift.
The study of the symmetric case (see definition after) has begun in \cite{Bed.Kab.14}. Here we do not want to restrict our study to the symmetric maps inside the bijective case. Nevertheless we restrict our study to a finite family of angles. 

Our method of investigation is close to the one introduced in \cite{Bed.Ca.11} for the outer billiard outside regular polygons. The main idea is to find a reasonable set, where we can consider the first return map and prove that it is conjugated to the initial map. This allows us to use substitutions in order to describe the language of the map.

\begin{figure}
\begin{center}

\includegraphics[width= 9cm]{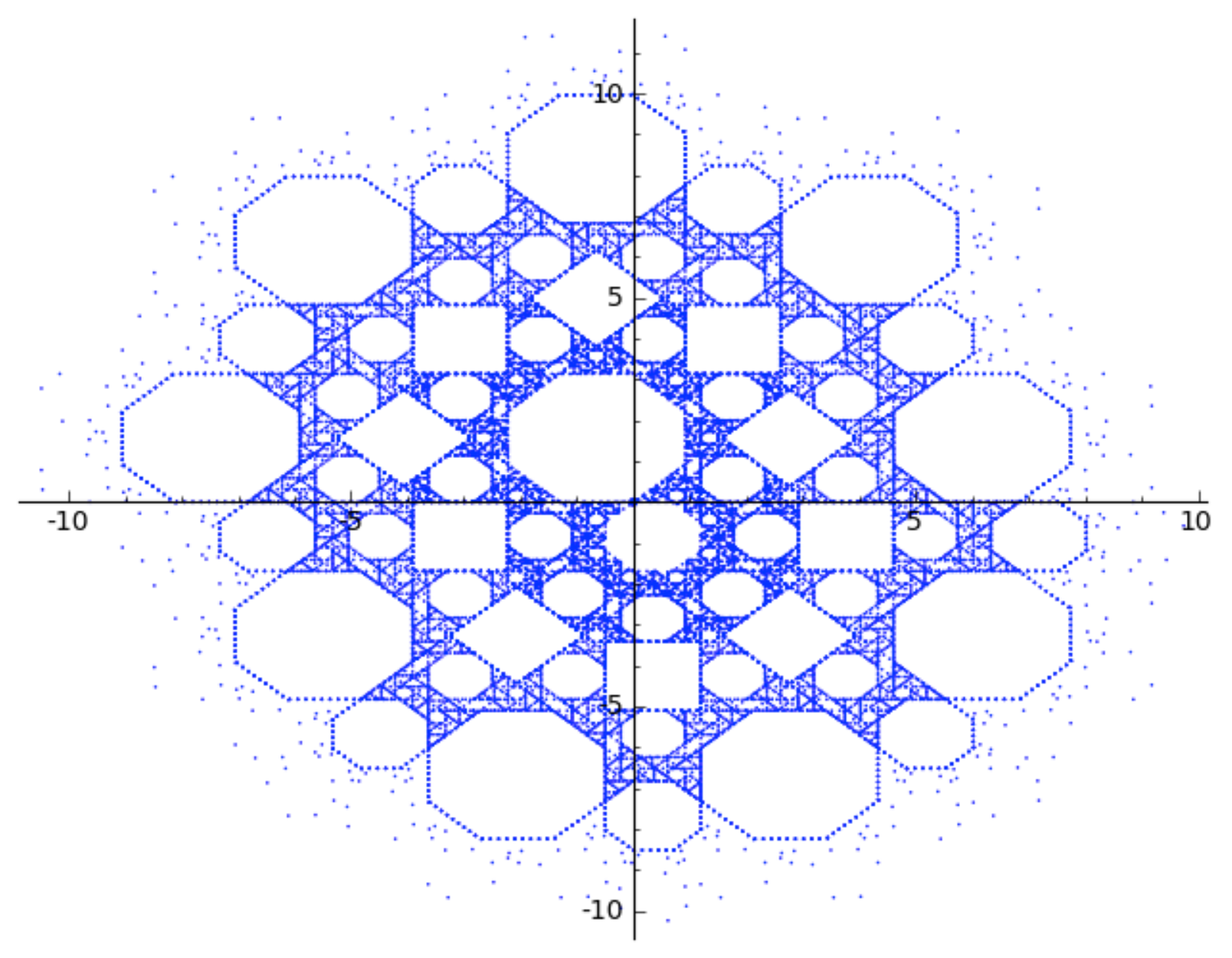}
\caption{Non symmetric map with $\sigma=\frac{1}{3}, \theta=\frac{1}{8}$.}\label{fig=nonsym=simple}
\end{center}
\end{figure}

\medskip

{\bf Keywords}: Piecewise isometries, substitutions, renormalization, words.
{\bf AMS classification}: 37A05, 37B10.\\
{\bf Thanks}: The second author would like to thanks EMS-Simons for Africa for its grant.
%%%%%%%%%%%%%%%%%%%%%%%%%%%%%%%%%%%%%%%%%%%%%%%%%%%%%%%%%%%%%%%%%%%%%%%%%%%%%%%%%%%%%%%%%%%%%%%%%%%%
\section{Definition of a piecewise isometry and its language}
\subsection{Basic notions on words}
We need to introduce some notions of symbolic dynamics, see \cite{Pyth.02}.
Let $\mathcal{A}$ be a finite set of symbols called  alphabet, a {\bf word} is a finite string of elements in $\mathcal{A}$, its length is the number of elements in the string. The set of all finite words over $\mathcal{A}$ is denoted $\mathcal{A}^*$. 
A (one sided) infinite sequence of elements of $\mathcal{A}$, $u=(u_n)_{n\in\mathbb{N}}$, is called an infinite word. We denote the infinite word by using concatenation $u=u_0\dots u_n\dots$
The infinite word $u$ is {\bf periodic} if there exists a finite word $v_0\dots v_n$ such that $u=v_0\dots v_n v_0\dots v_n v_0\dots v_n\dots$ Such an infinite word is denoted $v^\omega$. 
A word $v_0\dots v_k$ appears in $u$ if there exists an integer $i$ such that $u_{i}\dots u_{i+k}=v_0\dots v_{k}$. In this case, we say that $v$ is a factor of $u$. For an infinite word $u$, the {\bf language} of $u$ (respectively the language of length $n\in\mathbb N$)  is the set of all words (respectively all words of length $n$) in $\mathcal{A}^*$ which appear in $u$. We denote it by $L(u)$ (respectively $L_{n}(u)$). 

A {\bf substitution} is a morphism of the monoid $\mathcal{A}^*$. In general we define the substitution by the images of the elements of $\mathcal A$. A language is said to be a {\bf substitutive language} if there exists a finite set of substitutions such that each word in the language is a factor of some compositions of the substitution applied on one letter.

\begin{remark}
Since a substitution is a morphism, it is entirely described by the images of the letters of the alphabet. We will represent it as an array. In a column we read a letter and its image. For example, the next array describes the well known Fibonacci substitution defined over $\mathcal A=\{1,2\}$.
$$\begin{array}{c|c|c}
a&1&2\\
\hline
\sigma(a)&12&1
\end{array}$$
\end{remark}

%%%
\subsection{Subshift associated to the piecewise rotation}
Let $P_1,P_2$ be the two half-planes bounded by the discontinuity line of $T_{\theta,\sigma}$. 
Let $\phi:\mathbb{C}\mapsto \{1;2\}^\mathbb{N}$ be {\bf the coding map} where the image of a complex number $z$ is given by $\phi(z)=(u_n)_{n\in\mathbb{N}}$ such that $T^n_{\theta,\sigma}(z)\in P_{u_n}$ for all integer $n$. 
The image by the coding map of the points which have well defined orbit defines a subshift. 
This subshift is defined as 
$$\Sigma=\overline{\{\phi(z), z\in X\subset\mathbb C\}}.$$
The {\bf language} of $\Sigma$ is the set of finite words which appears in some sequence $u\in\Sigma$.
For an infinite word $u$ in $\Sigma$, a {\bf cell} is the set of points which are coded by this word:
$\{z\in \mathbb{C}, \phi(z)=u\}$.
Remark that the coding map fulfills the condition: the point $z$ has a periodic orbit implies that $\phi(z)$ is a periodic word.

Remark that every point $z\in\mathbb{C}$ has not a well defined orbit for $T$. Consider the set of complex numbers $z$ such that there exists an integer $n$ with $T^nz\in\mathbb{R}$. This set of points is called the {\bf set of discontinuity points}, but it is of zero Lebesgue measure and we can ignore it. In the following, we will only consider orbits of points outside this set. In Figure \ref{fig=nonsym=simple}, we can see one example of the discontinuity set. The polygons correspond to the periodic cells.
\subsection{Restriction of the study of the map}

Due to the formulation of the map $T_{\theta,\sigma}$, the centers of rotations are the points
\begin{equation}
z_1=\frac{e^{2i\pi\theta}(\sigma+1)}{1-e^{2i\pi\theta}},\quad z_2=\frac{e^{2i\pi\theta}(\sigma-1)}{1-e^{2i\pi\theta}}.
\label{centres}
\end{equation}
If $\sigma=0$ the map is called a {\bf symmetric map}.
The imaginary parts of the centers of rotations are equal to $\frac{\sigma\pm 1}{2\tan{\pi\theta}}$. In all our cases $\tan{\pi\theta}$ is a positive number, thus both of the two centers of rotations define a periodic point of $T_{\theta,\sigma}$ if and only if $\sigma\in[-1,1]$. A simple computation shows:
\begin{lemma}
There are three cases:
\begin{itemize}
\item The points $O_1$ and $O_2$ belong to the half-plane $Im(z)>0$, and $O_2$ is far away from the origin.
\item The centers of rotations belong to the same half plane, and $O_1$ is far away from the origin.
\item The two points belong to different half-planes. 

\end{itemize}

Moreover, there is a symmetry of the problem.

For all $\theta, \sigma$, let us denote by $S$ the map defined by $S(z)=-z$. Then, we have for all $z\in\mathbb C$:
$$S\circ T_{\theta,\sigma}\circ S=T_{\theta,-\sigma}$$
\end{lemma}
The proof is left to the reader.

\begin{corollary}
The cases $\sigma<0$ and  $\sigma>0$ are symmetric and can be deduced one from the other.
\end{corollary}
\begin{proof}
Consider the case $\sigma<0$. By preceding Lemma it is enough to work with the conjugate by $S$ of the map of parameter $-\sigma$. 
Then this conjugate is a piecewise isometry. It has for centers of rotations, the points $S(O_1), S(O_2)$. So we remark that the language of $S\circ T_{\theta,\sigma}\circ S$ is obtained from the language of $T_{\theta,\sigma}$ by the exchange of letters $0$ and $1$.
\end{proof} 
Thus, we can restrict the values of the parameter $\sigma$ to the interval $[0,+\infty)$.
 In \cite{Bed.Kab.14}, we gave a complete description of the bijective symmetric map for the angle living in a finite set.

%%%%%%%%%%%%%%%%%%%%%%%%%%%%%
%%%%%%%%%%%%%%%%%%%%%%%%%%%%%%
%%%%%%%%%%%%%%%%%%%%%%%%%%%%%%

\section{Results and overview of the paper}
\subsection{Results}
We consider the finite set $\theta\in\{\frac{1}{4},\frac{1}{3}, \frac{1}{8}\}$ and assume that $\theta$ belongs to this set. Then we give a description of the symbolic dynamics for any value of $\sigma$. Each of the following sections is devoted to one angle. The last one deals with the last point of the theorem. This shows the limit on the parameter $\sigma$ in the statement of one Theorem of Goetz-Quas, see \cite{Goet.Quas.09}.

\begin{Theorem}
Consider a bijective piecewise rotation of angle $\theta$ and parameter $\sigma$.

\begin{itemize}
\item If $\theta=\frac{1}{4}$, for every $\sigma\in\mathbb R_+$, the language of the piecewise rotation is substitutive and can be described explicitly. 

\item For the angle $\frac{1}{3}$, the language of the piecewise rotation can be described explicitly.

\item If $\theta=\frac{1}{8}$ we obtain:
\begin{itemize} 
\item For $\sigma=\frac{1}{3}$, the language is substitutive.
\item In the case $\sigma=1$, the language can be described by a substitution. 
\item If $\sigma=3$, then the language is substitutive.
\end{itemize}

\item For every $\theta$ in this finite set, there exists $\sigma>0$ such that an orbit made of periodic cells does not form a closed ring.
\end{itemize}
\end{Theorem}

 \subsection{General method}
 We want to briefly explain the method used to prove that the language of the piecewise rotation is a substitutive language.
 Assume we find a subset $\mathfrak C\subset \mathbb R^2$ such that the first return map of $T$ to $\mathfrak C$ is well defined.  It is a piecewise isometry by definition. If this map is conjugated by an homeomorphism to $T$, then the language is substitutive. We refer to \cite{Bed.Ca.11} Lemma 31, for a proof of this fact. In all the following we apply this method. The set $\mathfrak C$ will be a cone with an angle which is a multiple of $\theta$. The vertex of the cone will be chosen carefully. We refer to the different Figures for the different choices depending on the angle. 
 
 To obtain the first return map to $\mathfrak C$, we give in the different Figures the sets $T^k\mathfrak C, k=1,2\dots$ until they intersect $\mathfrak C$ again. Remark that we draw in dashed lines the intersections of $\mathfrak C$ with its iterates by $T$.
 %%%%%%%%%%%%%%%%%%%%%%%%%%%%%%
%%%%%%%%%%%%%%%%%%%%%%%%%%%%%%%
%%%%%%%%%%%%%%%%%%%%%%%%%%%%%%%
%%%%%%%%%%%%%%%%%%%%%%%%%%%%%%%
\section{Angle $\theta=\frac{1}{4}$}
The aim of this section is to describe the map for a non zero parameter $\sigma$. Equation $(1)$ shows that the two centers of rotations are given by complex numbers
$\frac{\sigma+1}{2}(-1+i), \frac{\sigma-1}{2}(-1+i).$ As explained in the discussion after Equation (\ref{centres}) we can restrict to the case $\sigma>0$.
If $0<\sigma<1$, then the two points $0^\omega$ and $1^\omega$ are periodic elements of the subshift, otherwise only $0^\omega$ exists.

\subsection{Case $\sigma< 1$}
We will show that the symbolic dynamics are the same for all maps with $0<\sigma<1$. The cells corresponding to the words $0^\omega$ and $1^\omega$ are two squares of different lengths $\sigma+1$ and $1-\sigma$, centered around the two centers of rotations. Thus, we can define a cone $\mathfrak C$ which has a support on one edge of the square and on the discontinuity. The vertex of the cone is a vertex of the square on the discontinuity line.
 We consider the first return map $\hat{T}$ to the cone, as in the symmetric cases, see right part of Figure \ref{nonsym-1} and \cite{Bed.Kab.14}.  This map has the following form defined on four pieces, see Figure \ref{nonsym-1} left part. We code this map with four letters $A, B, C$ and $D$.
 The link between the natural coding of $T$ and the coding of $\hat T$ is given by the morphism $\begin{array}{c|c|c|c|}A&B&C&D\\ \hline
 1221&12^31&12211 &12^31^2\end{array}.$

 \begin{proposition}
 There exists a substitution $\sigma_{4,s,0}$ defined over $\{A,B,C,D\}^*$ such that
 for every $\sigma\in (0,1)$, the language $L'$ of the dynamics of $\hat{T}$ is the set of factors of the periodic words of the form $z^\omega$ for $z\in Z$, where:
 $$Z=\bigcup_{n\in\mathbb{N}}\{\sigma_{4,s,0}^n(D), \sigma_{4,s,0}^n(C), \sigma_{4,s,0}^n(B) \}.$$ 
\end{proposition} 
 \begin{proof}
First, we compute the first return map of $\hat{T}$ to $A$. A simple computation shows that it is conjugated to $\hat{T}$.
Thus, the return words of $A,B,C,D$ indicate us that the dynamics is given by the substitution 
 $$\sigma_{4,s,0}:\begin{array}{c|c|c|c|} A&B&C&D\\ \hline  A&AB&AC&ABC\end{array}$$
 
The orbit of $A$ under $\hat{T}$ does not cover all the cone $\mathfrak C$. But, the pieces not covered correspond exactly to the cells of the periodic words $D^\omega, C^\omega$ and $B^\omega$. We deduce that the periodic points are images under the substitution of $D^\omega, C^\omega$ and $B^\omega$. 
 \end{proof}
As a corollary of the proof, since the substitution does not depend on $\sigma$, we deduce that the symbolic dynamics is the same for all values of $\sigma\in (0,1)$.
 
Remark that the case of a symmetric map is just a degenerated case where the periodic cells become squares instead of rectangles, see \cite{Bed.Kab.14} for a comparison.
 
 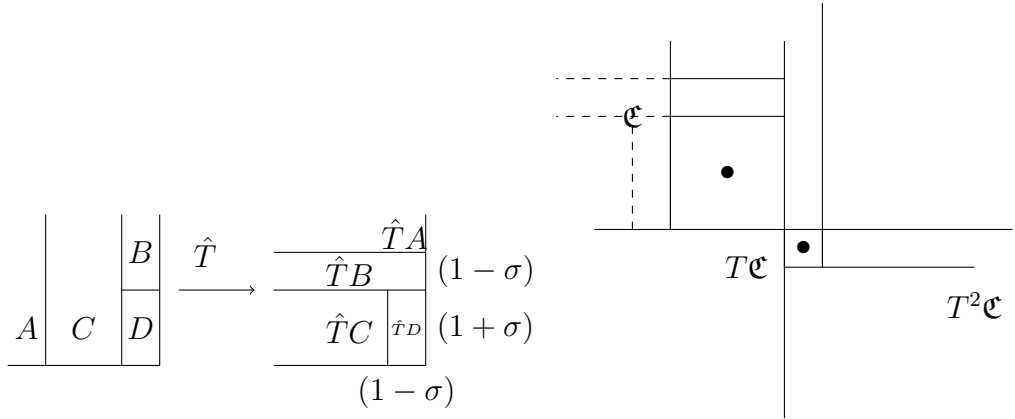
\begin{figure}[h]
 \begin{center}
\begin{tikzpicture}[scale=.5]
 \draw (-4,0)--(0,0)--(0,4);
 \draw (-3,0)--(-3,4);
 \draw (-1,0)--(-1,4);
 \draw (-1,2)--(0,2);
 \draw[->] (0.5,2)--(2.5,2);
 \draw (-.5,1) node{$D$};
 \draw (-2,1) node{$C$};
 \draw (-.5,3) node{$B$};
 \draw (-3.5,1) node{$A$};
 \draw (7,1) node[right]{$(1+\sigma)$};
\draw (7,2.5) node[right]{$(1-\sigma)$}; 
\draw (6.5,0) node[below]{$(1-\sigma)$};
\draw (6.5,1) node{\tiny$\hat T D$};
\draw (5,1) node{$\hat TC$};
\draw (5,2.5) node{$\hat TB$}; 
\draw (6.45,3.5) node{$\hat TA$};
\draw (3,0)--(7,0)--(7,4);
\draw (3,2)--(7,2);
\draw (3,3)--(7,3);
\draw (6,0)--(6,2);
\draw (1.2,3) node{$\hat T$};
\end{tikzpicture}
\begin{tikzpicture}[scale=.5]
\draw (-8,0)--(3,0);
\draw[dashed] (-7,0)--(-7,3);
\draw[dashed] (-6,3)--++(-3,0);
\draw[] (-3,4)--++(-3,0);
\draw[dashed] (-6,4)--++(-3,0);
\draw (-2,-1)--++(0,7);
\draw (-2,-1)--++(4,0);
\draw (-6,0)--(-6,3)--(-3,3);

\draw (-3,0)--(-3,-1)--(-2,-1)--(-2,0);
\draw (-4.5,1.5)node{$\bullet$};
\draw (-2.5,-.5)node{$\bullet$}; 
\draw (-6,0)--(-6,5);
\draw (-3,-5)--(-3,5);
\draw (-7,3)node{$\mathfrak{C}$}; 
\draw (-4,-1)node{$T\mathfrak{C}$}; 
\draw (2,-2) node{$T^2\mathfrak C$};
\end{tikzpicture}

\caption{The first return map for $\theta=\frac{1}{4}$ and $\sigma<1$. On the right part, the two centers of rotation are the black points. On the left part the description of $\hat{T}$.}\label{nonsym-1}
\end{center}
\end{figure}
 
 \subsection{Case $\sigma=1$}
 The proof is close to the previous one. The cells $B$ and $D$ just vanish :
 there exists an unique fixed point for $T$. The cell associated to it is a square of size $2$. 
 We consider the induction $\hat T$, as previously. Here the map is defined on two subsets coded by $A$ and $C$ (we keep these notations to be coherent with the previous case). This map is clearly self similar if we look at the first return map to $A$. The substitution is given by 
  $\sigma_{4,s,0}:\begin{array}{c|c|} A&C\\ \hline  A&AC\end{array}.$ Thus, we have proved the following result:
 
 \begin{proposition}
 The language is described by the words $\sigma_{4,s,0}^n(C^\omega), n\in\mathbb N$.
 \end{proposition}
 
\begin{figure}[h]

\begin{tikzpicture}[scale=.5]
 \draw (-4,0)--(-1,0)--(-1,4);
 \draw (-3,0)--(-3,4);
 \draw (-1,0)--(-1,4);
 \draw (-2,1) node{$C$};
  \draw (-3.5,1) node{$A$};
\draw (5,1) node{$\hat T C$};
\draw (6.4,3.5) node{$\hat TA$};
\draw (3,0)--(7,0)--(7,4);
\draw (3,2)--(7,2);
\draw[->] (0,2)--(2,2);
\draw (1,2.5) node{$\hat{T}$};
\end{tikzpicture}
\begin{tikzpicture}[scale=1]
\draw (-4,0)--(4,0);
\draw (-1,0)--(0,0)--(0,1)--(-1,1)--cycle;
\draw (-1,0)--++(0,2);
\draw (0,0)--++(0,-2);
\draw (0,0)--++(0,2);
\draw[dashed] (0,1)--++(-4,0);
\draw (0,0)node{$\bullet$};
\draw (-.5,.5)node{$\bullet$};
\draw (-2,-1) node{$T\mathfrak C$};
\draw (2,-1) node{$T^2\mathfrak C$};
\draw (2,1) node{$T^3\mathfrak C$};
\end{tikzpicture}
\caption{The first return map for $\theta=\frac{1}{4}$ and $\sigma=1$. On the right part, the orbit of $\mathfrak C$ under $T$.}

\end{figure}
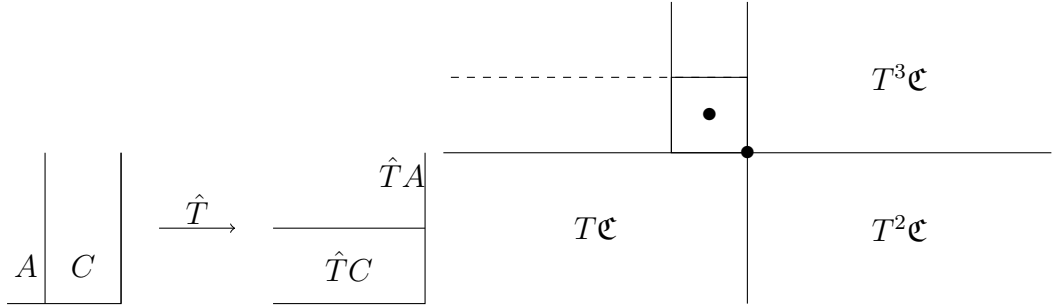

 \subsection{Case $\sigma>1$}
 The map $T$ has only one fixed point.
The infinite word $1^\omega$ is a periodic element of the subshift, the associated cell is a square. 

\begin{proposition}\label{prop14+1}
For every $\sigma\in (1,\infty)$, the language $L$ of the dynamics of $T$ is substitutive. 
\end{proposition} 
\begin{proof}
We consider the first return map $\hat{T}$ to the same cone as in the previous case. This map has the following form defined on three pieces, see Figure  \ref{nonsym+1} left part. Consider the first return map to $A$ for $\hat T$, and denote it $\hat T_A$. Assume we have $2n+1<\sigma<2n+3$ for some integer $n$. A simple computation shows that the map has the form given in Figure \ref{retret}. This map is defined on three subsets. The sets $A_2, A_3$ have for sizes $2n+3-\sigma, \sigma-2n+1$.
The return words depend on the parameter $\sigma$.  Then we obtain the morphism
$$\eta:\begin{array}{c|c|c}
A_1&A_2&A_3\\
\hline
A&ACB^{n-1}&ACB^{n}\\
\end{array}$$

Considering its first return map to $A_1$, we see this map is self-similar. The substitution is given by the morphism
$$\theta:\begin{array}{c|c|c}
A_1&A_2&A_3\\
\hline
A_1&A_1A_2&A_1A_3\\
\end{array}$$

 The orbit of the set $A_1$ under $\hat{T_A}$ does not cover all the piece labeled $A$. It remains several periodic cells. We deduce the language of the coding of $\hat{T_A}$. Here again, the orbit of $A$ under $\hat{T}$ does not cover all $\mathfrak C$. It remains a finite number of rectangles which correspond to the cells of periodic words. These words depend on $n$ and are of the form $(BC^k)^\omega, k\in\mathbb N$.
\end{proof}

\begin{remark}
First of all remark that several substitutions are needed in order to describe the language. Remark also that the substitutions involved in the description of the dynamics depends only on the value of $[\frac{\sigma-1}{2}]$.
\end{remark}
\begin{example}
For $\sigma=4$,  the orbit of the set $A_1$ under $\hat{T_A}$ does not cover all the piece named $A$. It remains several periodic cells coded by $A_2^\omega, (A_2A_3)^\omega$. Then we use the renormalisation to obtain the language of $\hat T_A$. Then we use the recoding to obtain the following result: the language of $\hat T$ is the set of factors of the periodic words of the form $z^\omega$ for $z\in Z$, where:

$$Z=\displaystyle\bigcup_{k\in\mathbb N}\{\eta\circ\theta (A_2)^k, \eta\circ\theta(A_2A_3)^k, (BC)^k, (BBC)^k, (BBBC)^k\}$$
\end{example}

\begin{figure}
\begin{center}
 \begin{tikzpicture}[scale=.5]
 \draw (-4,0)--(0,0)--(0,4);
 \draw (-4,1)--(0,1);
 \draw (-2,1)--(-2,4);
 \draw[->] (0.5,2)--(2.5,2);
 \draw (-.5,1.5) node{$C$};
 \draw (-1.5,.5) node{$B$};
 \draw (-3.5,1.5) node{$A$};
 \draw (7,1) node[right]{$2$};
 \draw (7,2.5) node[right]{$(\sigma-1)$};  
 \draw (5,1) node{$\hat T C$};
 \draw (5,2.5) node{$\hat T B$}; 
 \draw (6.3,3.5) node{$\hat T A$};
 \draw (3,0)--(7,0)--(7,4);
 \draw (3,2)--(7,2);
 \draw (3,3)--(7,3);
 \end{tikzpicture}
 \begin{tikzpicture}[scale=.25]
 \draw (-10,0)--(5,0);
 \draw (-6,0)--(-6,4)--(-2,4)--(-2,0);
 \draw (-4,2)node{$\bullet$};
  \draw (-3,1)node{$\bullet$};
  \draw (-6,0)--(-6,10);
  \draw (-2,0)--(-2,-10);
  \draw (-2,2)--(5,2);
  \draw (-2,0)--(-2,10);
  \draw (-4,4)--(-4,10);
  \draw[dashed] (-6,4)--++(-4,0);
    \draw[dashed] (-6,2)--++(-4,0);
  \draw (-7,2) node{$\mathfrak C$};
    \draw (-4,-2) node{$T\mathfrak C$};
 \end{tikzpicture}
\caption{The first return map for $\theta=\frac{1}{4}$ and $\sigma>1$.}\label{nonsym+1}
\end{center}
\end{figure}
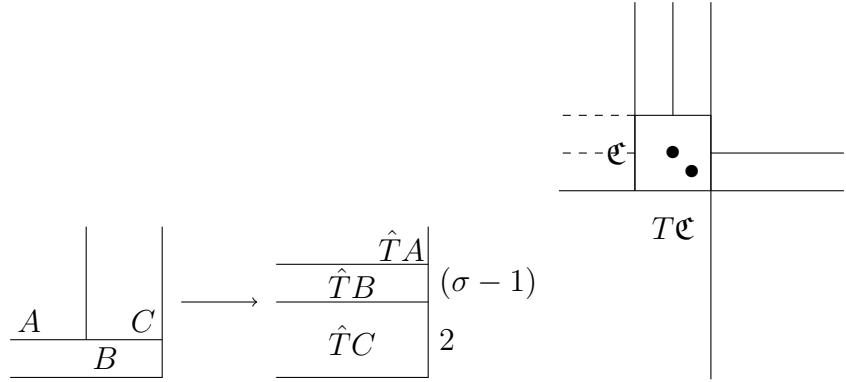

\begin{figure}
\begin{center}
 \begin{tikzpicture}[scale=.5]
 \draw (-4,0)--(0,0)--(0,4);
 \draw (-1,0)--(-1,4);
 \draw (-2,0)--(-2,4); 
 \draw (-2,1)--(-2,4);
 \draw[->] (0.5,2)--(2.5,2);
 \draw (-.5,.5) node{$A_3$};
 \draw (-1.5,.5) node{$A_2$};
 \draw (-3.5,.5) node{$A_1$};
\draw (5,2.5) node{$\tiny\hat T_A A_1$};
 \draw (5,1.5) node{$\tiny\hat T_A A_2$}; 
 \draw (5,.5) node{$\tiny\hat T_A A_3$};
 \draw (3,0)--(7,0)--(7,4);
 \draw (3,2)--(7,2);
 \draw (3,1)--(7,1);
 \draw (1,2) node[below]{$\hat T_A$};
 \end{tikzpicture}
\begin{tikzpicture}[scale=.5]
\fill[gray] (-4,0)--(-2,0)--(-2,3)--(0,3)--(0,5)--(-4,5)--cycle;
\draw (-2,0)--(0,0)--(0,3);
\draw (0,2)--(-2,2);
\draw (-1,0) node[below]{$2$};
\draw (0,1) node[right]{$2$};
\draw (0,2.5) node[right]{$\sigma-1$};
\end{tikzpicture}
 
\caption{On the left: The first return map of $\hat T$ to $A$ for $\theta=\frac{1}{4}$ and $\sigma>1$. On the right part in gray, the orbit of the piece $A$ under $\hat T$ in order to compute the first return map.}\label{retret}
\end{center}
\end{figure}
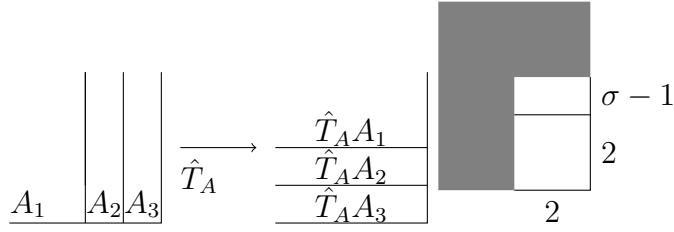

%%%%%%%%%%%%%%%%%%%%%%
%%%%%%%%%%%%%%%%%%%%%%
\section{Angle $\theta=\frac{1}{3}$}
\subsection{Calculus}
The centers of the rotations, due to Equation $(1)$, have the form 
$$z_{2}=-\frac{1}{\sqrt 3}e^{-i\pi/6}(\sigma-1), z_{1}=-\frac{1}{\sqrt 3}e^{-i\pi/6}(\sigma+1).$$ 
Thus, these points $O_1, O_2$ are on a line passing through the origin with an angle of $-\frac{\pi}{6}$. The $x$ coordinates of these points are $-\frac{\sigma-1}{2}, -\frac{\sigma+1}{2}$. We deduce that these two points are inside the half-plane $Im(z)>0$ if and only if $\sigma>1$. Remark also that $|z_{1}-z_{2}|=\frac{2}{\sqrt 3}$. 

\subsection{Case $\sigma>1$}
The point $O_1$ is the center of an equilateral triangle of side $\sigma+1$, with two vertices on the real line ($O$ and $z=-(\sigma+1)$). 
The point $O_2$ belongs to the segment $[O,O_1]$. We consider $\mathfrak C$ the cone of angle $\frac{2\pi}{3}$, which has for center the other vertex of the triangle on the real line. We compute the first return map to $\mathfrak C$. It is given by Figure \ref{return13}. The link with the natural coding is given by
$\begin{array}{c|c|c|} A&B&C\\ \hline 1221&121&122\end{array}.$

\begin{figure}
\begin{center}
\begin{tikzpicture}[scale=.25]
\draw (-10,0)--(0,0)--++(60:7);
\draw (-10,0)--(0,0)--++(60:2)--++(120:5);
\draw (-10,0)--(0,0)--++(60:2)--++(120:2)--++(60:5);

\draw[->] (3,1)--++(2,0);
\draw (4,2) node{\tiny$\hat T$};
\draw (-3,1) node{$B$};
\draw (0,5) node{\tiny$C$};
\draw (1,4) node{\tiny$A$};

\draw (11,1) node{$\hat TA$};
\draw (11,3) node{\tiny$\hat T B$};
\draw (17,7) node{\tiny$\hat TC$};

\draw (5,0)--(15,0)--++(60:2)--++(-7,0);
\draw (5,0)--(15,0)--++(60:5)--++(120:5);
\draw (5,0)--(15,0)--++(60:7);
\end{tikzpicture}
\begin{tikzpicture}[scale=.25]
\draw (-10,0)--(5,0);
\draw (-5,0)--(0,0)--++(120:5)--++(240:5);
\draw (-5,0)--++(60:10);
\draw (-3,0)--++(-120:5);
\draw (0,0)--++(120:-5);
\draw (2,0)--++(120:-5);
\draw (0,0)--++(120:3)--++(60:5);
\draw[dashed] (-2.5,4.3)--++(120:6);
\draw[dashed] (-5,0)--++(60:2)--++(-7,0);
\end{tikzpicture}
\caption{On the left part: the first return map for $\theta=\frac{1}{3}$ and $\sigma>1$. On the right part, how to find it by iterations of $T_{\theta,\sigma}$}\label{return13}
\end{center}
\end{figure}
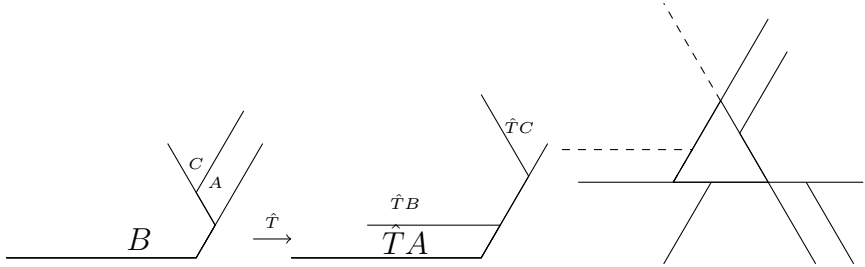

\begin{proposition}
Every point of $\mathfrak C$ has a periodic orbit under $\hat T$. The associated periodic cell is either an equilateral triangle or an hexagon. The cone $\mathfrak C$ is tiled by these two polygons.
\end{proposition}
\begin{proof}
The proof follows the same method as the one of Proposition \ref{prop14+1}.
Assume $3<\sigma<5$ (the proof is similar if $2n+1<\sigma<2n+3$ for some integer $n$). Then we have $2<\sigma-1<4<\sigma+1$.
Consider two equilateral triangles of size $5-\sigma, \sigma-3$. Then consider an hexagon with sides two by two parallel with the same lengths as the triangles. We can tile $\mathfrak C$ with these three tiles, as made in Figure \ref{13pavage}. Now, it is easy to see that each polygon has a periodic orbit under $\hat T$. Remark that the restriction of $\hat T$ to $B$ and to $C$ is a translation, and that the restriction to $A$ is a rotation of angle $\frac{2\pi}{3}$. For example, the codings of the two marked hexagons are given by
$$\bullet=(BA)^\omega, \circ=(BBCA)^\omega$$

\begin{figure}
\begin{center}
\begin{tikzpicture}[scale=.35]
\draw (-10,0)--(0,0);
\draw (-.5,0)--++(120:10);
\draw (-2.5,0)--++(120:10);
\draw (-4.5,0)--++(120:10);

\draw (0,0)--++(60:2)--++(-10,0);
\draw (0,0)--++(60:4)--++(-10,0);
\draw (0,0)--++(60:6)--++(-10,0);
\draw (0,0)--++(60:8)--++(-10,0);

\draw (0,0)--++(60:1.5)--++(120:7);
\draw (0,0)--++(60:3.5)--++(120:7);
\draw (0,0)--++(60:5.5)--++(120:7);

\draw (-2,0)--++(60:10);
\draw (-4,0)--++(60:10);
\draw (-6,0)--++(60:10);
\draw (-8,0)--++(60:10);

\fill (-.5,1) circle(.2);
\fill (.5,3) circle(.2);

\draw (-2.5,1) circle(.2);
\draw (-1.3,2.2) circle(.2);
\draw (-.5,4.5) circle(.2);
\draw (1.5,4.5) circle(.2);
\end{tikzpicture}
\caption{Dynamics of $\hat T$ for $\theta=\frac{1}{3}$ and $\sigma>1$.}\label{13pavage}
\end{center}
\end{figure}
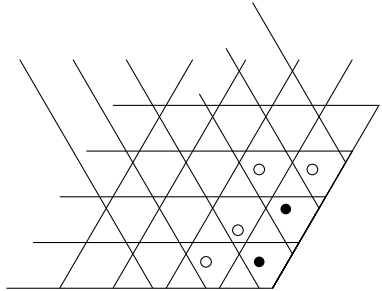
\end{proof}

\begin{remark}
The periodic tiling does not depend of the value of $\sigma$.
\end{remark}
%%%%%%%%%%%%%%%%%%%%%%%
%%%%%%%%%%%%%%%%%%%%%%%
\section{Angle $\theta=\frac{1}{8}$}

\subsection{Case $\sigma=1$}
\begin{proposition}\label{prop:s=1-octa}
The language of the piecewise rotation of angle $\frac{1}{8}$ and parameter $\sigma=1$ is substitutive.
\end{proposition}
\begin{proof}
In this case, one of the two centers of rotation is on the discontinuity line and is a vertex of the regular octagon which is the cell associated to the periodic word $1^\omega$.
Now, we consider the sector of angle $\frac{\pi}{4}$ with one boundary on the discontinuity line and one boundary supported by one edge of the octagon, see Figure \ref{fig:s=1}. Consider the first return map $\hat T$ of $T$ to this sector. 
A simple computation shows that $\hat{T}$ is given by the right part of Figure \ref{fig:s=1}. 
There are 4 return words: $A=12^41^3, B=12^41^4, C=12^41^5$ and $D=12^41^6$.
Now, we consider the first return map of $\hat{T}$ to the cell associated to $A$. This map is conjugated to $\hat{T}$.
The substitution is given by
$$\begin{array}{c|c|c|c}
A&B&C&D\\
\hline A&AB&ABB&ABBB\end{array}$$
\end{proof}

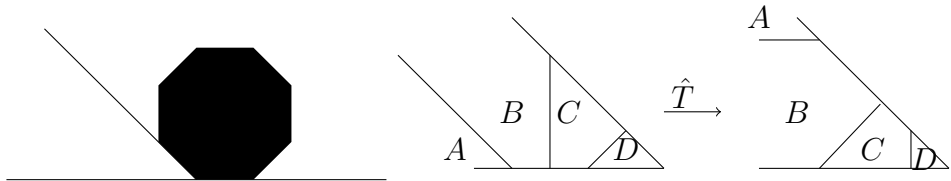
\begin{figure}[h]
\begin{center}
\begin{tikzpicture}[scale=.25]
\draw (-10,0)--(10,0);
\fill[]  (3,0)--(5,2)--(5,5)--(3,7)--(0,7)--(-2,5)--(-2,2)--(0,0);
\draw (0,0)--(-8,8);
\end{tikzpicture}
\begin{tikzpicture}[scale=.25]
\draw (20,0)--(30,0);
\draw (30,0)--(22,8);
\draw (28,2)--(28,0);
\draw (23.18,0)--(26.39,3.41);
\draw (23.18,6.81)--(20,6.81);

\draw[->] (15,3)--(18,3);
\draw (16,4) node{$\hat T$};

\draw (5,0)--(15,0);
\draw (15,0)--(7,8);
\draw (7,0)--(1,6);
\draw (11,0)--(13,2);
\draw (9,0)--(9,6);

\draw (4,1) node{$A$} ;
\draw (7,3) node{$B$} ;
\draw (10,3) node{$C$};
\draw (13,1) node{$\tiny D$};

\draw (20,8) node{$A$} ;
\draw (22,3) node{$B$} ;
\draw (26,1) node{$C$};
\draw (28.7,.5) node{$\tiny D$};

\end{tikzpicture}
\caption{ Piecewise rotation of angle $\frac{\pi}{4}$ and $\sigma=1$.}\label{fig:s=1}
\end{center}
\end{figure}

\subsection{Case $\sigma>1$}

We study the example $\sigma=3$. It corresponds to the case where the two centers are on the same half-plane. The point $O_2$ is at the center of a regular octagon, and the second one, $O_1$, is a vertex of this octagon such that it is the center of a two times bigger octagon, see Figure \ref{fig-2octa}. Remark that the image by the rotation of center $O_1$ of the vertex $A$ is the point $B$ such that the middle of $[AB]$ is a vertex of the small octagon. The cone $\mathfrak C$ is the same as in the other cases. We consider $\hat T$ the first return map of $T$ to this set.
\begin{proposition}
The language of the piecewise rotation of angle $\frac{1}{8}$ and parameter $\sigma=3$ is substitutive.
\end{proposition}

\begin{figure}
\begin{center}
\begin{tikzpicture}[scale=.25]
\draw (-10,0)--(10,0);
\draw[]  (3,0)--(5,2)--(5,5)--(3,7)--(0,7)--(-2,5)--(-2,2)--(0,0);
\draw (0,7) node[above]{$O_1$};
\draw[dashed] (0,7)--(3,0);
\draw[dashed] (0,0)--(3,7);
\draw (1.5,3.5) node[left]{$O_2$};
\draw (3,0)--++(5,5);
\draw[dotted] (0,7)--(7,4);
\draw (3,0)--++(-5,-5);
\draw (3,0) node[below]{$A$};
\draw (7,4) node[below]{$B$}; 
\end{tikzpicture}
\caption{Positions of the centers in the case $\theta=\frac{1}{4}, \sigma=3$.}\label{fig-2octa}
\end{center}
\end{figure}
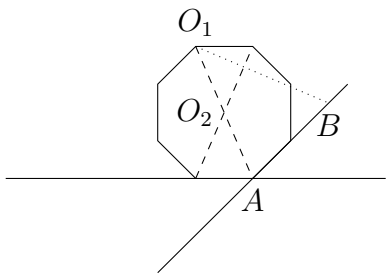

\begin{proof}
The first return map $\hat{T}$ to the cone $\mathfrak C$ is a piecewise isometry defined by Figure \ref{fig=r143}. 
It is defined on eight pieces. The return words of these pieces are given by
$$\begin{array}{c|c|c|c|c|c|c|c}
A&B&C&D&E&F&G&H\\
\hline 12^41^3&12^31^4&12^41^4&12^31^5&12^21^5&121^6&12^21^6&12^31^6\end{array}$$

Now we describe the first return map of $\hat T$ to $A$. 

This map is defined on four pieces, named $A_1,\dots, A_4$. The return words are 
$$\begin{array}{c|c|c|c}
A_1&A_2&A_3&A_4\\
\hline
A&ACB&ACBCB& ACBCBCB\end{array}$$

This map is  the same map as the one given in Proposition \ref{prop:s=1-octa}. We explain in Figure \ref{returntoA:143} how to obtain it. Now we can use the same arguments, as in the previous paper (see \cite{Bed.Kab.14}). We deduce that it defines a substitutive language given by 
$$\begin{array}{c|c|c|c}
A_1&A_2&A_3&A_4\\
\hline A_1&A_1A_2&A_1A_2A_2&A_1A_2A_2A_2\end{array}$$

Of course the orbit of $A$ under $\hat T$ does not fill $\mathfrak C$. Thus we need to work on the complement to finish the study. This can be done in a similar way to \cite{Bed.Kab.14}. We obtain that the language is also substitutive.
\end{proof}
\begin{figure}[]
\begin{center}
\begin{tikzpicture}[scale=.125]
\draw (-40,0)--(0,0);
\draw (0,0)--++(135:30);
\draw (-5,0)--++(135:5);
\draw (-3.535,0)--(-3.535,3.535);
\draw (0,0)--++(135:17.07)--++(-17.07,0)--++(135:10);
\draw (0,0)--++(135:17.07)--++(-25,0);
\draw (0,0)--++(135:17.07)--++(-12.07,0)--++(135:-17.07)--++(45:8.535);
\draw (0,0)--++(135:17.07)--++(-12.07,0)--++(135:-17.07)--++(45:7.2)--++(0,1.775);
\draw(-40,17) node{$A$};
\draw(-30,5) node{$B$};
\draw(-25,17) node{$C$};
\draw(-15,10) node{$D$};
\draw(-9,2) node{\tiny$E$};
\draw(-1.8,1) node{\tiny$F$};
\draw(-4.8,3) node{\tiny$G$};

\draw (10,0)--(50,0);
\draw (50,0)--++(135:40);
\draw (42.93,0)--++(0,7.07);
\draw (42.93,0)--++(0,3.535)--++(3.535,0);
\draw (22.93,0)--++(45:19.25)--++(-30,0);
\draw (22.93,0)--++(45:12.07)--++(10,0);
\draw (50,0)--++(135:34.14)--++(-15,0);
\draw (47.5,0)--++(45:1.775);

\draw(15,26) node{$\hat T A$};
\draw(15,19) node{$\hat T B$};
\draw(15,8) node{$\hat TC$};
\draw(35,5) node{$\hat T D$};
\draw(38,10) node{\tiny$\hat TE$};
\draw(44,5) node{\tiny$\hat TF$};
\draw(45,2) node{\tiny$\hat T G$};

\draw[->] (2,8)--++(3,0);
\draw (2,10) node{$\hat T$};
\end{tikzpicture}
\caption{Return map $\hat T$ for the non symmetric rotation of angle $\frac{\pi}{4}$ and $\sigma=3$.}\label{fig=r143}
\end{center}
\end{figure}
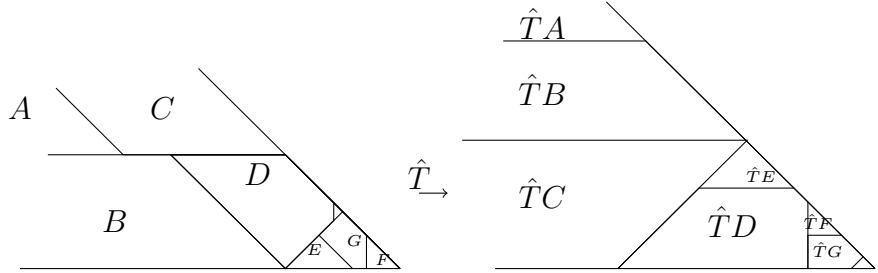

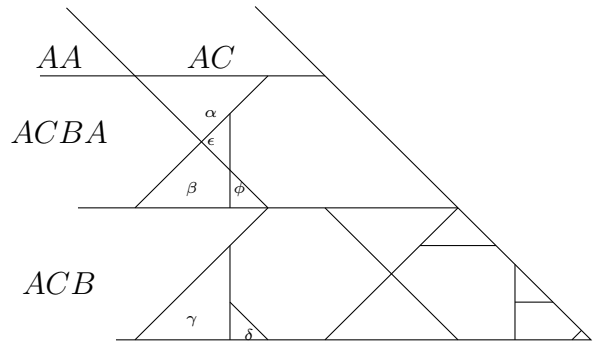
\begin{figure}
\begin{center}
\begin{tikzpicture}[scale=.25]
\draw (-25,0)--(0,0)--++(135:25);
\draw (-4,0)--++(0,4);
\draw (-4,2)--(-2,2);
\draw (-1,0)--++(45:.705);
\draw (-5,5)--++(-4,0);
\draw (-7,7)--(-14,0);
\draw (-7,7)--++(-20,0);
\draw (-14,14)--++(-15,0);
\draw (-14,7)--(-7,0);
\draw (-17,7)--++(135:15);
\draw (-17,7)--(-24,0);
\draw (-19,7)--++(0,5);
\draw (-17,14)--(-24,7);
\draw (-19,0)--++(0,5);
\draw (-19,2)--(-17,0);

\draw (-20,15) node{$AC$};
\draw (-28,15) node{$AA$};
\draw (-28,11) node{$ACBA$};
\draw (-20,12) node{\tiny $\alpha$};
\draw (-21,8) node{\tiny $\beta$};
\draw (-28,3) node{$ACB$};
\draw (-18.5,8) node{\tiny$\phi$};
\draw (-18,.3) node{\tiny$\delta$};
\draw (-20,10.5) node{\tiny $\epsilon$};
\draw (-21,1) node{\tiny $\gamma$};
\end{tikzpicture}
\caption{First return map of $\hat{T}$ to $A$ for $\theta=\frac{1}{8}$, $\sigma=3$. We denote $\alpha=ACBC$, $\beta=ACBCBA$, $\gamma=ACBCB$, $\delta=ACBCBCB$, $\epsilon=ACBCBC$ and $\phi=ACBCBCBA$}\label{returntoA:143}
\end{center}
\end{figure}

%%%%%

\subsection{Case $\sigma<1$}
We finish by the case $\sigma=\frac{1}{3}$, without all the details.
We consider the first return map to $\mathfrak C$. The structure of $\hat{T}$ is given by Figure \ref{fig1413}. This map is defined on $9$ pieces which are given by the coding
$$\begin{array}{c|c|c|c|c|c|c|c|c}
A_1&A_2&A_3&A_4&A_5&A_6&A_7&A_8&A_9\\
\hline
12^41^4&12^51^4&12^41^5&12^51^5&12^41^6&12^51^6&12^61^6&12^71^7&12^61^7
\end{array}$$
The end of the proof is similar to the other cases.

\begin{figure}
\begin{center}
\begin{tikzpicture}[scale=.15]
\draw (-35,0)--(15,0);
\fill (-6,0)--(-10,4)--(-10,10)--(-6,14)--(0,14)--(4,10)--(4,4)--(0,0);
\fill (0,0)--(-2,-2)--(-2,-5)--(0,-7)--(3,-7)--(5,-5)--(5,-2)--(3,0);
\draw (-6,0)--++(-30,30);
\draw (0,0)--++(20,20);
\draw (10,0)--++(20,20);
\draw (4,4)--++(0,25);
\draw[dashed] (-10,0)--(-10,4);
\draw[dashed] (-17,0)--(-17,3);
\draw[dashed] (-20,0)--++(10,10);
\draw[dashed] (-30,0)--++(12,12);
\draw[dashed] (-10,1)--(-8.5,2.5);
\draw[dashed] (-10,3)--++(-3.5,3.5);
\draw[dashed] (-20,14)--++(-15,0);
\draw[dashed] (-27,21)--++(-8,0);
\draw (0,0)--++(-15,-15);
\draw (-6,-3) node{$T\mathfrak C$};
\draw (3,-7)--++(10,10);
\draw (3,-7)--++(20,0);
\draw (5,-5)--++(0,10);
\draw (5,-2)--++(135:6);
\draw (11,11)--++(0,20);
\draw (4,6)--++(2,0);
\draw (7.5,7.5)--(4,11);
\draw (4,15)--(-1,15);
\draw (0,14)--++(135:20);
\draw (4,20)--++(135:10);
\draw (-7,21)--++(-20,0);
\draw (-6,14)--++(-14,0);
\draw (-7,14)--++(45:5);
\draw (-2,14)--++(0,2);
\draw (-11,14)--++(0,-5);
\draw (-10,8)--++(-2,0);
\draw (-16,14)--++(45:-3);
\draw (-9,14)--++(135:-3);
\draw (10,-3) node{$T^5\mathfrak C$};

\draw (-35,25) node{\tiny$A_1$};
\draw (-30,17) node{\tiny$A_2$};
\draw (-30,7) node{\tiny$A_3$};
\draw (-20,7) node{\tiny$A_4$};
\draw (-18,1) node{\tiny$A_5$};
\draw (-13,3) node{\tiny$A_6$};
\draw (-8,1) node{\tiny$A_9$};
\end{tikzpicture}
\caption{How to find the first return map to  the cone $\mathfrak C$ for $\sigma=\frac{1}{3}, \theta=\frac{1}{8}$. The set $T^5\mathfrak C$ is the first one which is splitted in two parts by the real line. }\label{fig1413}
\end{center}
\end{figure}
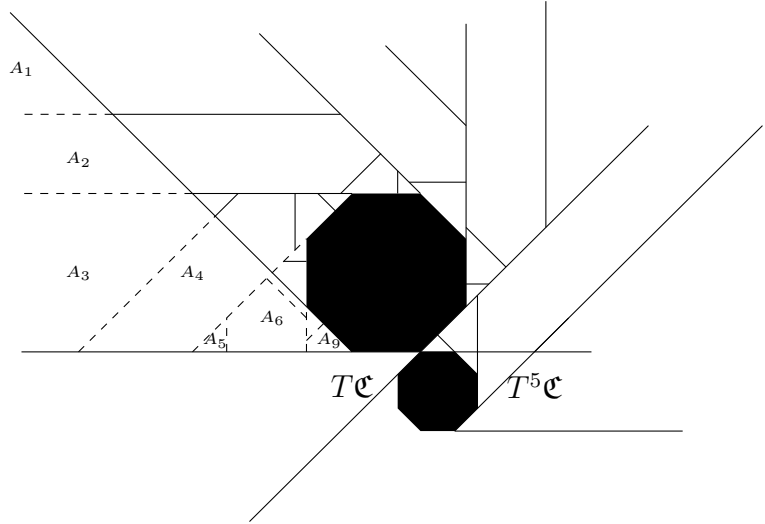

%%%%%%%%%
%%%%%%%%%
\section{Last part of the theorem}
The following proposition shows the limit in the theorem of Goetz-Quas since it was only valuable for small values of $\sigma$.

 \begin{proposition}
Consider the two following piecewise rotations: one of angle $\theta=\frac{1}{4}$ and $\sigma=3$ and one of angle $\frac{1}{8}$ and $\sigma=\frac{1}{3}$. Then
\begin{itemize}
\item There are several shapes and sizes for the cells of periodic words. 
\item A periodic cell does not form an invariant annulus with touching cells. 
\item The union of several periodic cells forms an invariant annulus.
\end{itemize}
\end{proposition}

\subsection{{Case $\theta=\frac{1}{4}, \sigma=3$}}
Consider the case  $\theta=\frac{1}{4}, \sigma=3$. Around $O_1$ there is an invariant square. The point $O_2$ is inside this square at the center of the square in dashed line. Then we consider the black squares. They represent the orbit $(12211)^\omega$. The little white square represent the orbit $(12^21^32^21^321^2)^\omega$. It is clear that the orbit of the black square does not form a closed ring.

\begin{tikzpicture}[scale=.25]
\draw (-10,0)--(10,0);
\draw (-3,0)--(-3,6)--(3,6)--(3,0);
\draw[dashed] (1,0)--(1,2)--(3,2);
\draw (-5,0)--(-5,6)--(-3,6);
\draw (-3,0)--(-3,-2)--(3,-2)--(3,0);
\draw (5,0)--(5,6)--(3,6);
\draw (-3,6)--(-3,8)--(3,8)--(3,6);
\draw (-5,4)--++(2,0);
\draw (-5,2)--++(2,0);
\draw (3,4)--++(2,0);
\draw (3,2)--++(2,0);
\draw (-1,-2)--(-1,-4)--(1,-4)--(1,-2)--cycle;
\draw (-1,10)--(-1,8)--(1,8)--(1,10)--cycle;
\fill (3,-2)--(5,-2)--(5,0)--(3,0)--cycle;
\fill (-1,-2)--(1,-2)--(1,0)--(-1,0)--cycle;
\fill (3,2)--(5,2)--(5,4)--(3,4)--cycle;
\fill (-1,6)--(1,6)--(1,8)--(-1,8)--cycle;
\fill (-5,2)--(-3,2)--(-3,4)--(-5,4)--cycle;
\draw (-7,2)--(-5,2)--(-5,4)--(-7,4)--cycle;
\draw (7,2)--(5,2)--(5,4)--(7,4)--cycle;
\draw (7,0)--(5,0)--(5,-2)--(7,-2)--cycle;
\draw (5,-2)--(3,-2)--(3,-4)--(5,-4)--cycle;
\draw (0,3) node{\tiny $0_1$};
\draw (2,1) node{\tiny $0_2$};
\end{tikzpicture}

\subsection{Case $\theta=\frac{1}{8}, \sigma=\frac{1}{3}$}

We refer to  Figure \ref{fig=nonsym=simple} for $\sigma=\frac{1}{3}$. The two cells of the orbits $1^\omega$ and $2^\omega$ are equal to some regular octagons of different sizes. Consider the words $(1^42^5)^\omega$ and $(2^41^5)^\omega$, they correspond to two periodic orbits made by two regular octagons images of the two first by translations. The union of these two octagons form an annulus around the origin. Now, consider the orbit $(1^52^5)^\omega$. It corresponds to the square on Figure \ref{fig=nonsym=simple}. 
Finally, the orbit  $(1^52^41^52^4)^\omega$ corresponds to the small octagons outside the square.

%%%%%%%%%%%%%%%%%%%%%%%%%%%%%%
%%%%%%%%%%%%%%%%%%%%%%%%%%%%%%
%%%%%%%%%%%%%%%%%%%%%%%%%%%%%%
\bibliographystyle{plain}
\bibliography{biblio-rotv3}
\end{document}